\documentclass[11pt]{amsart}
\usepackage[margin=3cm]{geometry}
\usepackage{amsmath,amssymb,amsthm}
\usepackage{tensor}
\usepackage{enumerate}
\usepackage{verbatim}
\usepackage{cite}
\usepackage{graphicx}
\usepackage{booktabs,tabularx,ragged2e}
\newcolumntype{Y}[1]{>{\centering\tiny}m{\dimexpr#1\textwidth-2\tabcolsep-1\arrayrulewidth\relax}}%-1\arrayrulewidth
\newcolumntype{Z}{@{}m{0pt}@{}} %Defines an empty column, used to add minimum space to each row

\renewcommand{\P}{\mathcal{P}}
\newcommand{\Q}{\mathcal{Q}}
\newcommand{\R}{\mathcal{R}}
\DeclareMathOperator{\Acc}{Acc}
\DeclareMathOperator{\Crit}{Crit}
\DeclareMathOperator{\Post}{Post}
\DeclareMathOperator{\Card}{Card}

\theoremstyle{plain}
\newtheorem{theorem}{Theorem}

\newtheorem{lemma}{Lemma}
\newtheorem{proposition}[lemma]{Proposition}
\newtheorem{corollary}[lemma]{Corollary}
\newtheorem*{theoremrecapitulate}{Theorem~\ref{th:main}'}
\theoremstyle{definition}
\newtheorem{definition}[lemma]{Definition}

\newtheorem{remark}[lemma]{Remark}

\begin{document}

\title[Finitely Generated Interval Maps]{Constant Slope Models for Finitely Generated Maps}
\author{Samuel Roth}
\address{Silesian University in Opava,\\Na Rybni\v{c}ku 626/1,\\746 01 Opava, Czech Republic}
\email{samuel.roth@math.slu.cz}

\begin{abstract}
We study countably monotone and Markov interval maps. We establish sufficient conditions for uniqueness of a conjugate map of constant slope. We explain how global window perturbation can be used to generate a large class of maps satisfying these conditions.
\end{abstract}
\maketitle

\section{Introduction}
\subsection{Motivation}\strut\\
From the topological conjugacy class of a topologically mixing \textbf{piecewise monotone} interval map $f$, we can choose out one distinguished representative, namely, the constant slope model%
\footnote{i.e., a conjugate map which is piecewise affine with slope $\pm\lambda$ on each piece, for some $\lambda>0$.}%
. It exists, it is unique%
\footnote{ignoring the trivial reversal of orientation: if $g$ is a constant slope model, then so is $x\mapsto-g(-x)$.}%
, and the logarithm of its slope tells us the topological entropy of $f$ \cite{AM,P}.

The situation is more complicated for topologically mixing \textbf{countably monotone} interval maps.  Constant slope models need not exist.  They need not be unique.  And the exponential of the topological entropy gives only a lower bound for the constant slope \cite{BB, BS, MR}.

In this paper we introduce \textbf{finitely generated} maps as a compromise -- they are countably monotone, but offer much of the simplicity of piecewise monotone maps.
%Finitely generated maps are also topologically mixing maps, produced from piecewise monotone maps using countable Markov partitions through a generating mechanism called global window perturbation, introduced by J. Bobok and H. Bruin \cite{BB}.
We will prove (Theorem~\ref{th:main}) that a finitely generated map has at most one constant slope model, and if it exists, then the logarithm of the slope gives the topological entropy.

\subsection{Outline}\strut\\
The central arguments of this paper may be summarized as follows.  In Section~\ref{sec:defs} we define global window perturbations and finitely generated maps. In Section~\ref{sec:dyn} we start with two well-known properties of a piecewise monotone topologically mixing interval map:
\begin{enumerate}[(i)]
\item\label{it:leo} the map is locally eventually onto (Lemma~\ref{lem:pmleo}), and
\item\label{it:rep} all periodic orbits are weakly repelling (Lemma~\ref{lem:perrepel})%
\footnote{In fact, they are repelling, but we need the weaker version of this property defined in \S\ref{sec:defs} to get the equivalence (ii)$\iff$(ii').}.
\end{enumerate}

Next, we give an interpretation of these two properties in symbolic dynamics. If a topologically mixing interval map admits a countable Markov partition with only finitely many accumulation points, then properties~\eqref{it:leo} and~\eqref{it:rep} are equivalent to the following two properties:
\begin{enumerate}[(i)$\iff$(i')]
\item\label{it:leo'} the transition matrix is locally eventually onto (Lemma~\ref{lem:leoleo}), and
\item\label{it:rome} the transition graph has a finite Rome (Theorem~\ref{th:rome}).
\end{enumerate}

Global window perturbation preserves properties~(i') and~(ii'), and thus finitely generated maps inherit these properties (Corollaries~\ref{cor:fgleo} and~\ref{cor:fgrep}).

In Section~\ref{sec:eig} we show the significance of properties~(i') and~(ii'), as regards nonnegative eigenvectors of a transition matrix:
\begin{enumerate}[(i')]
\item $\implies$ all nonnegative eigenvectors are summable (Proposition~\ref{prop:leosum}), and
\item $\implies$ there is at most one nonnegative eigenvector (Theorem~\ref{th:path} and Corollary~\ref{cor:fgunique}).
\end{enumerate}

In Section~\ref{sec:main} we exploit the link between nonnegative eigenvectors and constant slope models in order to prove our main result, Theorem~\ref{th:main}.

In Section~\ref{sec:exa} we develop Theorem~\ref{th:mixing}, which offers a sufficient criterion for a map to be topologically mixing. Then we apply the theorem to construct an illustrative example of a finitely generated map.

\section{Definitions}\label{sec:defs}

\subsection{Interval Maps}~\\
We work in the space $C([0,1])$ of continuous maps $f:[0,1]\to[0,1]$.  The \emph{critical points} of a map $f\in C([0,1])$ are
\begin{equation*}
\Crit(f):=\{0,1\}\cup\{x \, | \, f \text{ is not strictly monotone on any neighborhood of }x\}.
\end{equation*}
$\Crit(f)$ is always closed; $f$ is \emph{piecewise monotone} if $\Crit(f)$ is finite. The \emph{postcritical set} of $f$ is
\begin{equation*}
\Post(f)=\overline{\cup_{n=0}^\infty f^n(\Crit f)}.
\end{equation*}
A set $P$ is \emph{invariant} for $f$ if $f(P)\subset P$. $\Post(f)$ is always closed and invariant.

\subsection{Markov Partitions}\strut\\
For interval maps $f$ with postcritical set countable (i.e., finite or countably infinite), we use two kinds of \emph{Markov partitions}:
\begin{itemize}
\item A \emph{taut partition} is a countable, closed, invariant set $P\supseteq\Crit(f)$.
\item A \emph{slack partition} is a countable, closed, invariant set $P\supseteq f(\Crit(f))$.
\end{itemize}
Of course, the minimal taut partition is $\Post(f)$ itself, but we will often use larger (finer) partition sets.  Note also that a slack partition can always be refined into a taut partition by adjoining $\Crit(f)$.

Given a Markov partition $P$, the \emph{partition intervals} $I,J,\ldots$ are the connected components of $[0,1]\setminus P$. The set of partition intervals will be denoted by the corresponding script letter $\P$, and we will use the term \emph{partition} to refer to either $P$ or $\P$, depending on the need.

\subsection{Global Window Perturbations and Finitely Generated Maps}\strut\\[-1em]

\begin{definition}\label{def:window}
Fix a map $f$ with a countable postcritical set, and choose a taut Markov partition $P$.  We say that $g$ is a \emph{global window perturbation}%
\footnote{Slack partitions and global window perturbations were first defined in~\cite{BB} in a slightly more restrictive sense -- for example, they allowed only finitely many critical points in each partition interval.}
of $(f,P)$ if
\begin{enumerate}[(i)]
\item\label{it:slack} $P$ is a slack partition for $g$, and
\item\label{it:match} $f(\overline{I})=g(\overline{I})$ for each $I\in\P$.
\end{enumerate}
\end{definition}

\begin{definition}\label{def:fg}
A \emph{finitely generated map} $g$ is a global window perturbation of a pair $(f,P)$, where $f$ has finite critical set, $P$ has finite accumulation set, and $f$, $g$ are both topologically mixing.
\end{definition}

%Figure~\ref{fig:fg} illustrates the production of a finitely generated map.  The first frame shows the choice of $f$ and $P$. In the second frame we draw ``windows'' along the graph of $f$, i.e. rectangles of the form $\overline{I}\times f(\overline{I})$, $I\in\P$.  We obtain $g$ by modifying $f$ within each window, replacing monotone behavior by piecewise monotone behavior. As required, we stay inside the windows, choose turning values in $P$, and make $g$ topologically mixing (cf. Corollary~\ref{cor:constructfg}).
%
%\begin{figure}[htb!!]
%\includegraphics[width=0.2\textwidth]{P.png}
%\includegraphics[width=0.2\textwidth]{windows.png}
%\includegraphics[width=0.2\textwidth]{g.jpg}
%\caption{The production of a finitely generated map}\label{fig:fg}
%\end{figure}

\section{Dynamic Properties of Finitely Generated Maps}\label{sec:dyn}

This section records two well-known properties of topologically mixing piecewise monotone maps, and then explores to what extent these properties are inherited by finitely generated maps.

\subsection{Two properties of piecewise monotone maps}\strut\\
We distinguish two notions of repelling behavior for a finite set $P$ invariant under $f$.  We say that $P$ is
\begin{itemize}
\item \emph{repelling} if for some open set $U\supset P$, for all $x\in U\setminus P$, for some $n\in\mathbb{N}$, $f^n(x)\notin U$.
\item \emph{weakly repelling} if for some open set $U\supset P$, for all $x\in U$, for some $n\in\mathbb{N}$, $f^n(x)\in P$ or $f^n(x)\notin U$.
\end{itemize}

\begin{lemma}\label{lem:pmleo}
\cite[Proposition 2.34]{Ru} A piecewise monotone topologically mixing interval map is locally eventually onto.
\end{lemma}

\begin{lemma}\label{lem:perrepel}
A piecewise monotone topologically mixing interval map has all periodic orbits repelling.
\end{lemma}
Rather than finding Lemma~\ref{lem:perrepel} in the literature, we simply give a short proof.
\begin{proof}
Fix a periodic orbit $P=\{p_1,\cdots,p_n\}$ of the map $f$ and fix $\epsilon>0$ small enough that $U=\cup_{i=1}^n (p_i-\epsilon,p_i+\epsilon)$ has the following property: $U\cap\Crit(f)\subseteq P$.  If $x\in U$ has its orbit contained in $U$, then by monotonicity, so does the whole interval connecting $x$ to the nearest point $p_i\in P$. This contradicts the mixing property. 
\end{proof}

\subsection{The effect of global window perturbation in symbolic dynamics}\strut\\
The following lemma allows us to do symbolic dynamics with slack partitions.

\begin{lemma}\label{lem:components}
Suppose $f\in C([0,1])$ has a (taut or slack) Markov partition $P$. Then for each $J\in\P$, $f^{-1}(J)$ has finitely many connected components. Each of them is contained in a partition interval and is mapped homeomorphically onto $J$.
\end{lemma}
\begin{proof}
Let $U$ be a component of $f^{-1}(J)$. Then $f(U)$, being a subset of $J$, does not meet $P$. It follows that $U\cap P=\emptyset$ and $U\cap\Crit(f)=\emptyset$, i.e. $U$ is contained in a partition interval and is mapped homeomorphically onto its image. Maximality of $U$ as a connected subset of $f^{-1}(J)$ implies that the endpoints of $U$ both map to endpoints of $J$, which together with monotonicity, gives surjectivity $f(U)=J$.

Now suppose there is an infinite sequence $U_1, U_2, \ldots$ of distinct connected components of $f^{-1}(J)$. Fix two distinct points $y,y'\in J$, and for each $i\in\mathbb{N}$ find their respective preimages $x_i,x'_i\in U_i$. Disjointness of the open intervals $U_i$ implies that their lengths are decreasing to zero, so $|x_i-x'_i|\to 0$. By compactness, we may find a subsequence $x_{i_j}$ converging to some point $x\in[0,1]$. Then also $x'_{i_j}$ converges to $x$, so continuity gives both $f(x)=y$ and $f(x)=y'$, a contradiction.
%
%Left as an exercise.
\end{proof}

Given a map $f$ with a (taut or slack) Markov partition $P$, the corresponding \emph{transition matrix} $A=A(f,P)$ is the countable matrix with rows and columns indexed by $\P$ and entries
\begin{equation*}
A_{IJ} = \Card (I \cap f^{-1}(x)), \quad x\in J;
\end{equation*}
by Lemma~\ref{lem:components}, its entries are well-defined (independent of $x$) natural numbers (for a taut partition they are zeros and ones).  Likewise, we define the \emph{transition graph} $\Gamma=\Gamma(f,P)$ as the countable directed graph with vertices indexed by $\P$ and with $A_{IJ}$ distinct arrows pointing from $I$ to $J$.  We will write $I\to J$ to express the existence of at least one arrow pointing from $I$ to $J$.

\begin{lemma}\label{lem:same}
Let $g$ be a global window perturbation of $(f,P)$. Then 
\begin{equation*}
A(f,P)_{IJ}=0 \iff A(g,P)_{IJ}=0.
\end{equation*}
Similarly,
\begin{equation*}
I\to J \textnormal{  in } \Gamma(f,P) \iff I\to J\textnormal{  in }\Gamma(g,P).
\end{equation*}
\end{lemma}
\begin{proof}
This follows immediately from Definition~\ref{def:window}~\eqref{it:match}.
\end{proof}

Thus, global window perturbation preserves the location of zeros in the transition matrix, but may replace some 1's with larger natural numbers.  Similarly, both transition graphs have the same paths (as sequences of vertices) -- global window perturbation preserves the location of arrows but may change their multiplicities. This simple observation has profound consequences for the dynamics of finitely generated maps.

\subsection{Finitely generated maps inherit the locally eventually onto property}

Suppose $P$ is a (taut or slack) Markov partition for an interval map $f$. It's easy to see that if $f$ is topologically mixing, then $A=A(f,P)$ is
\begin{itemize}
\item \emph{irreducible}: for all indices $I,J$ there is $n>0$ with $A^n_{IJ}>0$, and
\item \emph{aperiodic}: for each index $I$, $\gcd\{n|A^n_{IJ}>0\}=1$.
\end{itemize}
If $f$ has the stronger property of being locally eventually onto, then $A(f,P)$ is also
\begin{itemize}
\item \emph{locally eventually onto}: for each index $I$ for some $n$ for all $J$, $A^n_{IJ}>0$.
\end{itemize}

The converse also holds:

\begin{lemma}\label{lem:leoleo}
Let $f\in C([0,1])$ be topologically mixing with a Markov partition $P$. Then $f$ is locally eventually onto if and only if $A(f,P)$ is locally eventually onto.
\end{lemma}
\begin{proof}
We prove only the implication $(\Leftarrow)$. First, suppose there exist at least three partition intervals. Fix an open set $U\subset[0,1]$. Choose a partition interval $I$ with neither $0$ nor $1$ as an endpoint. Find $n$ such that $f^n(U)$ meets both of the open sets $(0,\min I)$ and $(\max I,1)$. By the intermediate value theorem, $f^n(U)\supset \overline{I}$. Now find $n'$ such that $A(f,P)^{n'}_{IJ}\geq1$ for all $J\in\P$. Then $f^{n'}(\overline{I})$ is a closed set containing every partition interval. Therefore $f^{n+n'}(U)=[0,1]$, as required.

Next suppose there exist only one or two partition intervals. Then by Lemma~\ref{lem:components}, $f$ is piecewise monotone, so by Lemma~\ref{lem:pmleo}, $f$ is locally eventually onto.
\end{proof}
\begin{corollary}\label{cor:fgleo}
A finitely generated map is locally eventually onto.
\end{corollary}
\begin{proof}
Let $g$ be finitely generated from $(f,P)$. By Lemma~\ref{lem:same}, $A(g,P)$ is locally eventually onto if and only if $A(f,P)$ is. The result follows in light of Lemma~\ref{lem:pmleo}.
\end{proof}

\subsection{Finitely generated maps inherit (weakly) repelling periodic orbits.}\strut\\
We offer a symbolic characterization of the property of having weakly repelling periodic orbits.

\begin{definition}
Given a countable directed graph $\Gamma$, a \emph{Rome} is any subset $\R$ of the vertices of $\Gamma$ such that every path $I_0\to I_1\to \cdots$ in $\Gamma$ visits $\R$.
\end{definition}

\begin{theorem}\label{th:rome}
Suppose $f\in C([0,1])$ is topologically mixing and has a (taut or slack) Markov partition $P$ with a finite accumulation set. Then the transition graph $\Gamma(f,P)$ contains a finite Rome if and only if every periodic orbit of $f$ is weakly repelling.
\end{theorem}
\begin{corollary}\label{cor:fgrep}
A finitely generated map has all periodic orbits weakly repelling
\end{corollary}
\begin{proof}
At the symbolic level, global window perturbation preserves the existence of a finite Rome -- Lemma~\ref{lem:same}. In light of Lemma~\ref{lem:perrepel} and Theorem~\ref{th:rome}, the result follows.
\end{proof}

Before proving Theorem~\ref{th:rome}, we gather together the necessary lemmas. The first is more or less a standard exercise in symbolic dynamics; it links orbits of the interval map with paths (itineraries) in the transition graph.

\begin{lemma}\label{lem:itineraries}
Suppose $f\in C([0,1])$ is topologically mixing with a Markov partition $P$.
\begin{itemize}
\item If $y\in[0,1]$ is a point whose orbit does not meet $\Acc P$, then there is a path $I_0\to I_1\to \cdots$ in $\Gamma(f,P)$ such that $f^n(y)\in\overline{I_n}$ for each $n\in\mathbb{N}$.
\item For each path $I_0\to I_1 \to \cdots$ in $\Gamma(f,P)$ there is a point $y\in[0,1]$ such that $f^n(y)\in \overline{I_n}$ for each $n\in\mathbb{N}$.
\end{itemize}
\end{lemma}

\begin{proof}
Let $y$ be a point with orbit disjoint from $\Acc P$. By the Markov property, it suffices to find a sequence of partition intervals such that $f^n(y)\in \overline{I_n}$ and
\begin{equation}\label{meets}
f(I_n)\cap I_{n+1}\neq\emptyset
\end{equation}
for all $n\in\mathbb{N}$. We define the sequence recursively. Since $y\notin\Acc P$ we can find $I_0\in\P$ with $y\in\overline{I_0}$. Now suppose we have chosen $I_n$. If $f^{n+1}(y)\notin P$, then we take $I_{n+1}$ to be the connected component of $[0,1]\setminus P$ containing $f^{n+1}(y)$ and continuity of $f$ gives~\eqref{meets}. Otherwise, $f^{n+1}(y)\in P\setminus \Acc(P)$. Then $f(I_n)$ is a nondegenerate interval with $f^{n+1}(y)$ in its closure. Therefore $f(I_n)$ meets at least one of the partition intervals with endpoint $f^{n+1}(y)$ -- we take that partition interval for $I_{n+1}$.

Now suppose $I_0\to I_1\to\cdots$ is a path in $\Gamma(f,P)$. Define
\begin{gather*}
J_n=I_0 \cap f^{-1}(I_1) \cap \cdots \cap f^{-n}(I_n)\\
K_n=\overline{I_0} \cap f^{-1}(\overline{I_1}) \cap \cdots \cap f^{-n}(\overline{I_n})
\end{gather*}
The sets $J_n$ are nonempty by Lemma~\ref{lem:components}. The sets $K_n\supset J_n$ are compact, nonempty and $K_{n+1}\subset K_n$. Therefore the nested intersection $\cap K_n$ contains at least one point $y$.
\end{proof}

\begin{lemma}\label{lem:accinv}
Suppose $f\in C([0,1])$ is topologically mixing and admits a (taut or slack) Markov partition $P$. Then $\Acc P$ is invariant.
\end{lemma}
\begin{proof}
Fix $p\in\Acc P$, a neighborhood $V$ of $f(p)$, and find a neighborhood $U$ of $p$ which maps into $V$.  Then $P\cap U$ is infinite.  Now consider two cases.

First, suppose $\Crit(f)\cap U$ is infinite. Then choose an interval $(a,b)\subset U$ whose endpoints are critical points. Since $f$ is mixing, it has no flat spots, and therefore $f(a)\neq f(b)$. But $P$ contains all critical values, so we've found two distinct points in $V\cap P$. Thus $(P\cap V)\setminus\{f(p)\}\neq\emptyset$. Since $V$ was arbitrary, this shows that $f(p)\in\Acc P$.

Second, suppose $\Crit(f)\cap U$ is finite. Then $f(P\cap U)\subset P\cap V$ is infinite, so again $(P\cap V)\setminus\{f(p)\}\neq\emptyset$. Since $V$ was arbitrary, this shows that $f(p)\in\Acc P$.
\end{proof}

\begin{proof}[Proof of Theorem~\ref{th:rome}]
$(\Leftarrow)$
First we assume that $f$ has weakly repelling periodic orbits and construct a finite Rome. By Lemma~\ref{lem:accinv}, the finite set $\Acc P$ is forward invariant. Therefore it decomposes into finitely many periodic orbits $P_1,\ldots,P_k$ as well as a (possibly empty) set of preperiodic points $Q$.  By hypothesis, we can find open weakly repelling neighborhoods $U_i\supset P_i$.  By continuity, after shrinking these open neighborhoods we may assume that their images are bounded away from each other and from $Q$,
\begin{equation}\label{disjoint}
f(U_i)\cap U_j=\emptyset,\text{ for } i\neq j, \quad\text{and}\quad q\notin\overline{f(U_i)}, \text{ for }q\in Q,i=1,\ldots,k.
\end{equation}
It may happen that a point in $P_i$ is an endpoint of a partition interval (when it is only a ``1-sided'' accumulation point). But by shrinking the sets $U_i$ even further, we may assume that
\begin{equation}\label{notrap}
\text{if } I\in\P \text{ and } \overline{I}\cap P_i \neq\emptyset, \text{ then }\overline{I}\not\subset U_i, \text{ for }i=1,\ldots,k.
\end{equation}

For each preperiodic point $q\in Q$ find $n_0,j\in\mathbb{N}$ so that $f^{n_0}(q)\in P_j$, and let $V_q$ be a neighborhood of $q$ so small that
\begin{equation}\label{fallin}
f^{n_0}(V_q)\subseteq U_j \quad \text{and} \quad V_q\cap f(U_i)=\emptyset, \text{ for }i=1,\ldots,k.
\end{equation}

Let $U$ be the union of all the sets $U_i$ and $V_q$ for $i=1,\ldots,k$ and $q\in Q$.  Let $\R=\{\,I\in\P \,|\, \overline{I}\not\subset U\,\}$.  We claim that $\R$ is a finite Rome.  Finiteness of $\R$ follows because $U$ is an open set containing all the accumulation points of $P$.

To see that $\R$ is a Rome, consider an infinite path $I_0\to I_1\to \cdots$ in $\Gamma(f,\P)$.  By Lemma~\ref{lem:itineraries} there is a point $x\in[0,1]$ with $f^n(x)\in \overline{I_n}$ for all $n\in\mathbb{N}$.  If $I_0\in\R$, then there is nothing to prove.  Otherwise $x\in U$. By~\eqref{fallin} there is $n_0\geq0$ with $f^{n_0}(x)$ in one of the sets $U_i$. Since $U_i$ is weakly repelling there is some minimal number $n\geq n_0$ such that $f^n(x)\notin U_i$ or $f^n(x)\in P_i$. In the second case, by~\eqref{notrap}, $I_n\in\R$. In the first case, $f^n(x)\in f(U_i)\setminus U_i$, so by~\eqref{disjoint} and~\eqref{fallin}, $f^n(x)\notin U$, and again we have $I_n\in\R$.

$(\Rightarrow)$ Now we assume the presence of a finite Rome $\R$ and prove that $f$ has repelling periodic orbits. Suppose $x_0\mapsto x_1\mapsto \cdots \mapsto x_{n-1}\mapsto x_0$ is a periodic orbit, and write $X=\{x_0,\ldots,x_{n-1}\}$. Let
\begin{gather*}
\P_X=\{I\in\P\,|\,\overline{I}\cap X\neq\emptyset\},\\
\Q_X=\{J\in\P\,|\, f(J)\supset I\text{ for some }I\in\P_X\}.
\end{gather*}
$\P_X$ is finite since a point belongs to the closure of at most two partition intervals. $\Q_X$ is finite by Lemma~\ref{lem:components}. Let $\R'=\R\cup\P_X\cup\Q_X$. It is another finite Rome. Take
\begin{gather*}
Y=\left(\bigcup_{I\in\R'}\partial I\right)\cup \left(\bigcup_{I\in\R'}\partial f^{-1}(I)\right) \cup \Acc(P) \cup X,\\
\epsilon=\frac12\min\{|x-y| \,|\, x\in X, y\in Y, x\neq y\},\text{ and}\\
\delta\in(0,\epsilon) \text{ such that } |x-x'|<\delta \implies |f(x)-f(x')|<\epsilon.
\end{gather*}
Notice that $Y$ is finite by the finiteness of $\R'$, $\Acc(P)$, $X$, and by Lemma~\ref{lem:components}. Thus $\epsilon>0$. Then $\delta$ exists by the uniform continuity of $f$. Our repelling neighborhood is
\begin{equation*}
U=\bigcup_{i=0}^{n} U_i,\quad \text{where }U_i=B_{\delta}(x_i)=\{y\,|\,|y-x_i|<\delta\}
\end{equation*}
The choice of $\delta$ gives $U$ the following five properties:
\begin{enumerate}[(i)]
\item for $i\neq j$, $U_i\cap U_j=\emptyset$,
\item for all $i$, $f(U_i) \cap U \subset U_{i+1 \text{ mod }n}$,
\item $U\cap\Acc P \subset X$,
\item for all $I\in\R'$, $U \cap \partial I \subset X$, and
\item for all $I\in\R'$, $U \cap \partial f^{-1}(I) \subset X$.
\end{enumerate}

We claim that $U$ is a weakly repelling neighborhood of $X$. Suppose to the contrary that there exists $y\in U$ such that for all $k\in\mathbb{N}$, $f^k(y)\in U\setminus X$. For the remainder of the proof, all subscripts on $x$'s and $U$'s will be understood modulo $n$. By property~(i) there is a unique $i$ such that $y\in U_i$, and by property~(ii) we have $f^k(y)\in U_{i+k}$ for each natural number $k$. By property~(iii) the orbit of $y$ is disjoint from $\Acc P$, so by Lemma~\ref{lem:itineraries} there is a path $I_0 \to I_1 \to \cdots$ in the transition graph $\Gamma(f,P)$ with $f^k(y)\in\overline{I_k}$ for all $k$. Since $\R'$ is a finite Rome, there are infinitely many $k$ with $I_k\in\R'$. Let $V_k$ denote the half-open interval 
\begin{equation*}
V_k=\begin{cases} [f^k(y),x_{i+k}),&\text{if }f^k(y)<x_{i+k}\\(x_{i+k},f^k(y)],&\text{if }x_{i+k}<f^k(y)\end{cases}
\end{equation*}
We claim that for each $k\in\mathbb{N}$, $f|_{V_k}$ is monotone. This implies that $f^k(V_0)=V_k\subset U$ for all $k$, which contradicts topological mixing.  In fact, to finish the proof, it is enough to establish implications~\eqref{inPx}-\eqref{noCrit} below -- then our claim follows from a backwards induction argument, remembering that there are arbitrarily large natural numbers $k$ with $I_k\in\R'$.
\begin{gather}
\label{inPx} I_k\in\R'\implies V_k\subset I_k \text{ and } I_k\in\P_X.\\
\label{inR'} I_{k+1}\in\P_X\implies I_k\in\R'.\\
\label{noCrit} I_{k+1}\in\P_X\implies V_k\cap \Crit(f)=\emptyset.
\end{gather}

In this paragraph we prove implication~\eqref{inPx}. Assume $I_k\in\R'$. By property~(iv), $\partial I_k \cap U_{i+k}\subset\{x_{i+k}\}$. Thus $I_k$ either contains all of $U_{i+k}$ or else ``half'' of it, extending from the centerpoint $x_{i+k}$ past one endpoint. In particular, this proves $I_k\in\P_X$. Now $f^k(y)$ is in $\overline{I_k}$ and is in $U_{i+k}$ and is not equal to $x_{i+k}$, so we may conclude that $f^k(y)\in I_k$. This shows that $V_k\subset I_k$.

In this paragraph we prove implication~\eqref{inR'}. Assume $I_{k+1}\in\P_X$. The arrow $I_k\to I_{k+1}$ in the transition graph means that $f(I_k)\supset I_{k+1}$, so $I_k\in\Q_X\subset\R'$.

In this paragraph we prove implication~\eqref{noCrit}. Assume $I_{k+1}\in\P_X$. Now $V_k$ is connected and meets $f^{-1}(I_{k+1})$ but not $\partial f^{-1}(I_{k+1})$ -- since $f^k(y)\in V_k$ and by property~(v). This implies that $V_k$ does not meet the complement of $f^{-1}(I_{k+1})$. Thus $f(V_k)\subset I_{k+1}$. Since $I_{k+1}\cap P=\emptyset$ and all critical values are in $P$, it follows that $V_k\cap\Crit(f)=\emptyset$.
\end{proof}

\section{Eigenvectors for Finitely Generated Maps}\label{sec:eig}

The main goal of this section is to establish summability and uniqueness of the eigenvectors associated with the transition matrix of a countably generated map.

Let $A$ be any nonnegative matrix with rows and columns indexed by a countable set $\P$.  We use the terms \emph{eigenvalues} and \emph{eigenvectors} of $A$ to refer only to the \textbf{nonnegative} solutions to the system of equations $\sum_J A_{IJ}v_J = \lambda v_I$, $I\in\P$.

\subsection{Summability follows from the locally eventually onto property.}\strut\\[-1em]
\begin{proposition}\label{prop:leosum}
Let $A$ be a matrix with countable index set $\P$ and entries from $\mathbb{N}$. If $A$ is locally eventually onto, then every eigenvector of $A$ is summable.
\end{proposition}
\begin{proof}
Fix $I\in\P$ and find $n\in\mathbb{N}$ such that for all $J\in\P$, $A^n_{IJ}>0$. Then each $A^n_{IJ}\geq1$, so for the eigenvector $v$ we have
\begin{equation*}
\sum_J v_J \leq \sum_J A^n_{IJ} v_J = \lambda^n v_I < \infty.
\end{equation*}
\end{proof}
\begin{corollary}\label{cor:fgsummable}
Let $g$ be finitely generated from $(f,P)$. Then every eigenvector of $A(g,P)$ is summable.
\end{corollary}
\begin{proof}
This follows in light of Lemma~\ref{lem:leoleo} and Corollary~\ref{cor:fgleo}.
\end{proof}

\subsection{Uniqueness follows from the presence of a finite Rome.}\strut\\
If a nonnegative matrix $A$ is irreducible and aperiodic, then the \emph{Perron value} of $A$ is the limit $\lambda_A=\lim \sqrt[n]{A^n_{IJ}}$, for fixed (arbitrary) $I,J$ from the index set.  Following Vere-Jones \cite{VJ}, if $\lambda_A<\infty$, then we call $A$
\begin{itemize}
\item \emph{recurrent} if the sum $\sum_n A^n_{IJ} \lambda^{-n}_A$ diverges, and 
\item \emph{transient} if the sum converges.
\end{itemize} 
The Perron value is known to be a lower bound for the eigenvalues of $A$.  In the recurrent case, it is an eigenvalue and it carries exactly one eigenvector, up to scaling, called the Parry eigenvector (see, eg. \cite{K}).

\begin{definition}\label{def:transeig}
An eigenvalue $\lambda$ of an irreducible aperiodic nonnegative matrix $A$ is called a \emph{transient eigenvalue}
if $A$ is Vere-Jones transient or if $\lambda > \lambda_A$.
\end{definition}

Thus, recurrent matrices may have transient eigenvalues.  The meaning of transience is simply that $\sum_n A^n_{IJ} \lambda^{-n}<\infty$ for some (every) pair $I,J$.

Note that Definition~\ref{def:transeig} is nonstandard. We introduce the terminology for the simple reason that transient eigenvalues can be used to define transient stochastic Markov chains. This is the main idea in the proof of Theorem~\ref{th:path}.

\begin{definition}
A \emph{simple path to infinity} in a countable directed graph $\Gamma$ is a path $I_0\to I_1 \to \cdots$ with all vertices distinct.  
\end{definition}

\begin{remark}\label{rem:incompatible}
Observe that no countable directed graph contains both a simple path to infinity and a finite Rome.  These two features are incompatible.
\end{remark}

Our next theorem gives a simple necessary condition for existence of a transient eigenvalue in terms of the structure of the corresponding directed graph.  In this theorem, multiplicity of arrows plays no role and we can allow $A$ to have arbitrary nonnegative entries. Thus, we simply define $\Gamma$ to be the graph with an arrow $I\to J$ if and only if $A_{IJ}>0$.  This theorem is an adaptation of a result by T.E. Harris \cite[Theorem 1]{H} from the theory of stochastic Markov chains.

\begin{theorem}\label{th:path}
Let $A$ be an irreducible aperiodic nonnegative matrix, and let $\Gamma$ be the corresponding countable directed graph.  If $A$ admits a transient eigenvalue, then $\Gamma$ contains a simple path to infinity.
\end{theorem}

\begin{proof}
Denote the transient eigenvalue by $\lambda$, the corresponding eigenvector by $v$.  Denote by $\P$ the index set for the rows and columns of $A$ (it is the vertex set of $\Gamma$).  Consider a \emph{stochastic} Markov chain with state space $\P$ and transition probabilities $p_{IJ}=\frac{A_{IJ}v_J}{\lambda v_I}$;~ it is clear that the row sums $\sum_J p_{IJ}$ are $1$ as required.  Moreover, $p_{IJ}>0$ if and only if $A_{IJ}>0$, so that (with probability 1) each path from the stochastic Markov chain is also a path in $\Gamma$.

The reader may easily verify by induction that the $n$-step transition probabilities are given by $p^{(n)}_{IJ}=\frac{A^n_{IJ}v_J}{\lambda^n v_I}$.  Transience of the eigenvalue $\lambda$ then implies that our stochastic Markov chain is transient in the stochastic sense, i.e., $\sum_n p^{(n)}_{IJ}<\infty$ for each pair $I,J$.  Thus the probability that a sample path returns to the same state infinitely often is zero -- see, eg., \cite[XIII.2 Definition 2 and XV.5 Definition 2]{F}.  So choose from the full measure set a sample path $x=x_0\to x_1\to\cdots$ which visits each state at most finitely many times.  Let $y$ be the path with $y_0=x_0$ and defined inductively by setting $y_{i+1}$ equal to the state which appears immediately after the last occurrence of $y_i$ in the path $x$.  Then $y$ is a simple path to infinity in $\Gamma$.
\end{proof}

\begin{corollary}\label{cor:fgunique}
Let $g$ be finitely generated from $(f,P)$. Then $A(g,P)$ has at most one nonnegative eigenvector. It exists if and only if $A(g,P)$ is recurrent and its eigenvalue is the Perron value $\lambda_{A(g,P)}$.
\end{corollary}
\begin{proof}
$\Gamma(g,P)$ contains a finite Rome -- see Theorem~\ref{th:rome} and Corollary~\ref{cor:fgrep}. Therefore it has no simple path to infinity -- see Remark~\ref{rem:incompatible}. We may conclude that $A(g,P)$ has no transient eigenvalues. The corollary then follows from the definition of a transient eigenvalue, coupled with the fact that the Perron value of a recurrent matrix carries exactly one nonnegative eigenvector.
\end{proof}

\section{Constant Slope Models for Finitely Generated Maps}\label{sec:main}

An interval map $f$ has \emph{constant slope} $\lambda$ if $|f'(x)|=\lambda$ for $x\not\in\Crit(f)$.  A \emph{constant slope model} of $f$ is a topologically conjugate map $g$ with constant slope.  To avoid trivial non-uniqueness, we take only conjugacies $g=\psi\circ f\circ\psi^{-1}$ by homeomorphisms $\psi:[0,1]\to[0,1]$ which preserve orientation.

\subsection{Main theorem.}\strut\\
Having studied the dynamical properties and eigenvectors of finitely generated maps, it is a very easy matter to state and prove our main theorem.

\begin{theorem}\label{th:main}
A finitely generated map $g$ has at most one constant slope model. It exists if and only if $g$ is Vere-Jones recurrent. The constant slope $\lambda$ is necessarily given by the entropy, $\lambda=\exp h(g)$.
\end{theorem}

A word of explanation is needed. \cite[Proposition 9]{BB} says that for a topologically mixing map with a countable post-critical set, the Vere-Jones class of the transition matrix is independent of the choice of the Markov partition, and is thus a property of the underlying interval map. This justifies our use of the language \emph{$g$ is Vere-Jones recurrent} as a condition in the theorem.

\begin{proof}[Proof of Theorem~\ref{th:main}]
There are two relevant facts from the literature applying to a topologically mixing map $g$ with a slack Markov partition $P$ and with finite entropy $h(g)<\infty$. First, J. Bobok  \cite[Theorem 3]{BB} showed that there is a bijection
\begin{equation}\label{bijection}
\left\{ \parbox{15em}{\centering summable eigenvectors of $A(g,P)$,\\ up to scaling} \right\}
\leftrightarrow
\left\{ \text{constant slope models for $g$} \right\},
\end{equation}
and the bijection carries eigenvalues to slopes. Second, J. Bobok and H. Bruin \cite[Proposition 8]{BB} showed that the Perron value corresponds the topological entropy $h(g)=\log \lambda_{A(g,P)}$.

In our situation we have $g$ finitely generated from a pair $(f,P)$. By Corollary~\ref{cor:fgunique} $A(g,P)$ has at most one eigenvector, namely, the Parry eigenvector, and it exists if and only if $g$ is recurrent. Its eigenvalue is the Perron value, and hence corresponds to the entropy.  Summability of the eigenvector follows from Corollary~\ref{cor:fgsummable}.  This proves the result for finite entropy maps.

If $h(g)=\infty$, then the proof of \cite[Proposition 8]{BB} shows that the Perron value $\lambda_{A(g,P)}$ is also infinite. In this case it makes no sense to speak of the Vere-Jones classification. Moreover, there cannot be any constant slope model, because a Lipschitz continuous interval map always has finite entropy~\cite{KH}.
\end{proof}

In fact, we have proved something slightly stronger. If we replace the hypothesis \emph{$g$ is finitely generated} with the individual properties of a finitely generated map, we get the following results.

\begin{theoremrecapitulate}
If $g\in C([0,1])$ is topologically mixing, $\Acc(\Post(g))$ is finite, and all periodic orbits of $g$ are weakly repelling, then $g$ has at most one constant slope model, and the slope $\lambda$ is necessarily given by the entropy, $\lambda=\exp h(g)$.

If $g\in C([0,1])$ is locally eventually onto, $\Post(g)$ is countable, $h(g)<\infty$, and $g$ is Vere-Jones recurrent, then $g$ has a constant slope model with slope $\lambda$ given by the entropy, $\lambda=\exp h(g)$.
\end{theoremrecapitulate}

\begin{proof}
The proof is essentially the same, but the first statement uses Theorem~\ref{th:rome}, Remark~\ref{rem:incompatible} and Theorem~\ref{th:path} in place of Corollary~\ref{cor:fgunique}, while the second statement uses Lemma~\ref{lem:leoleo} and Proposition~\ref{prop:leosum} in place of Corollary~\ref{cor:fgsummable}.
\end{proof}

\subsection{Example of a finitely generated map.}\strut\\

We begin by constructing an example of a finitely generated map. Let $f$ be the piecewise affine interval map whose graph connects the dots
\begin{equation*}
(0,1), \left(\tfrac13,0\right),\left(\tfrac12,\tfrac12\right),\text{ and }(1,0).
\end{equation*}
The following set $P$ is a taut Markov partition for $f$,
\begin{equation*}
P=\left\{0,\tfrac13,\tfrac12,1\right\} \cup \left\{\tfrac12\left(1-\tfrac{1}{3^n}\right)\,|\,\,n\in\mathbb{N}\right\}.
\end{equation*}
Let $g$ be the countably piecewise affine interval map whose graph connects the dots
\begin{equation*}
(0,1),  \left(\tfrac13,0\right),\left(\tfrac12,\tfrac12\right), \text{ and } \left(\tfrac12+\tfrac12\cdot \tfrac{1}{3^n}, \tfrac12 \right), \left(\tfrac12+\tfrac12\cdot\tfrac{2}{3^{n+1}},\tfrac12-\tfrac12\cdot\tfrac{1}{3^n}\right), \,\, n\in\mathbb{N}.
\end{equation*}
Figure~\ref{fig:fgexample} illustrates the process by which $g$ is finitely generated from $(f,P)$. The first frame shows $f$ and $P$. In the second frame we draw ``windows'' along the graph of $f$, i.e. rectangles of the form $\overline{I}\times f(\overline{I})$, $I\in\P$. The third frame shows $g$ -- it fills the same windows and has $P$ as a slack Markov partition.

\begin{figure}[htb!!]
\includegraphics[width=0.25\textwidth]{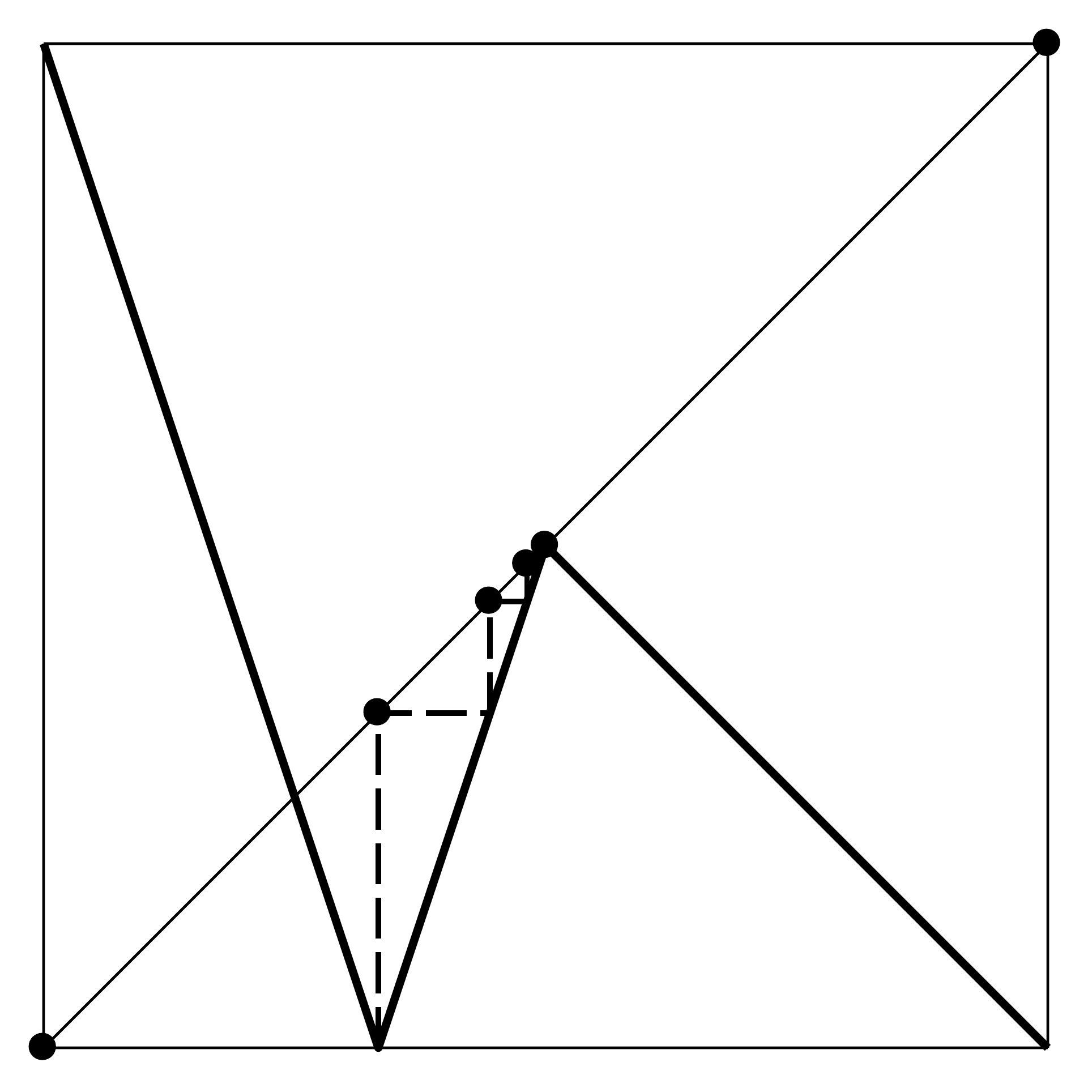}
\includegraphics[width=0.25\textwidth]{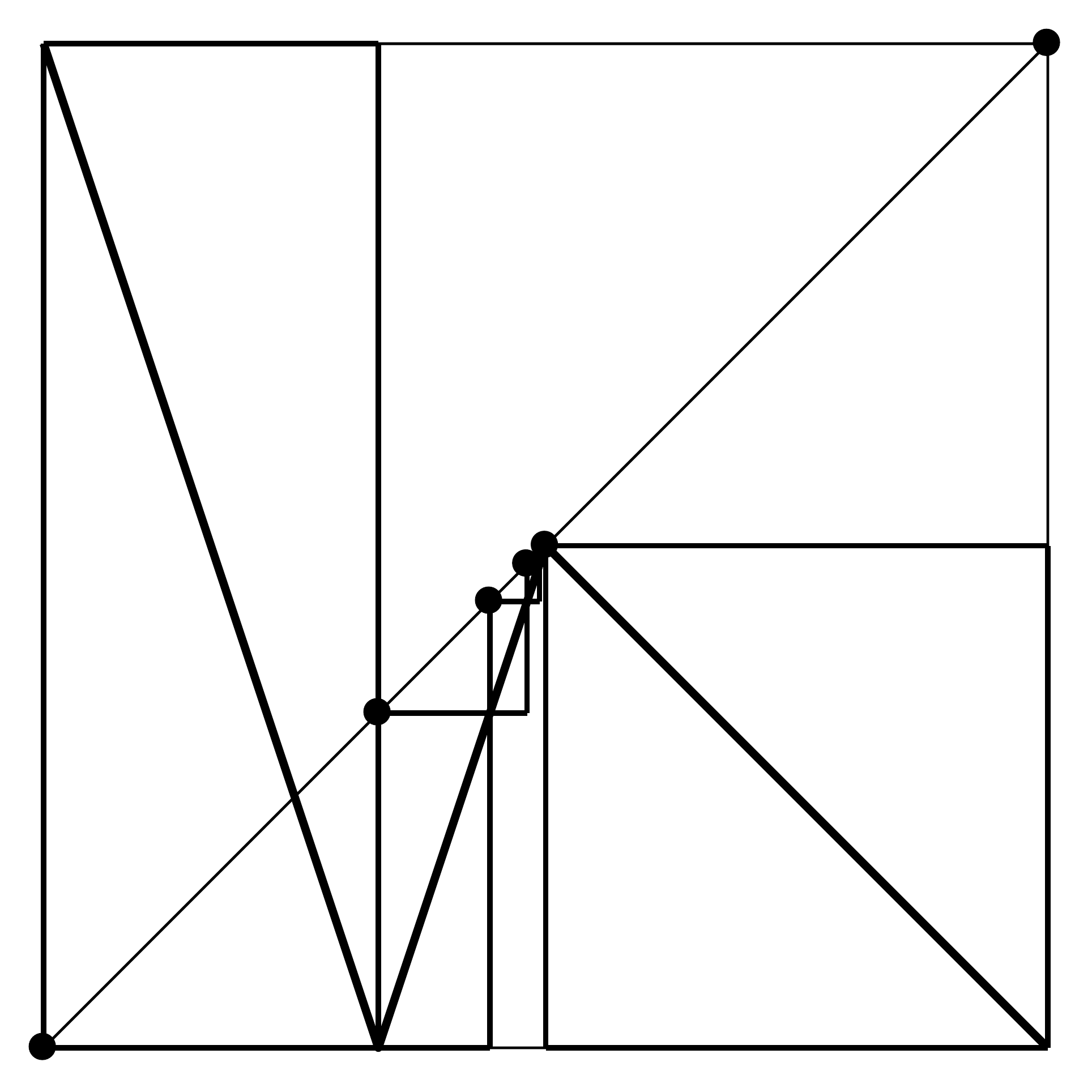}
\includegraphics[width=0.25\textwidth]{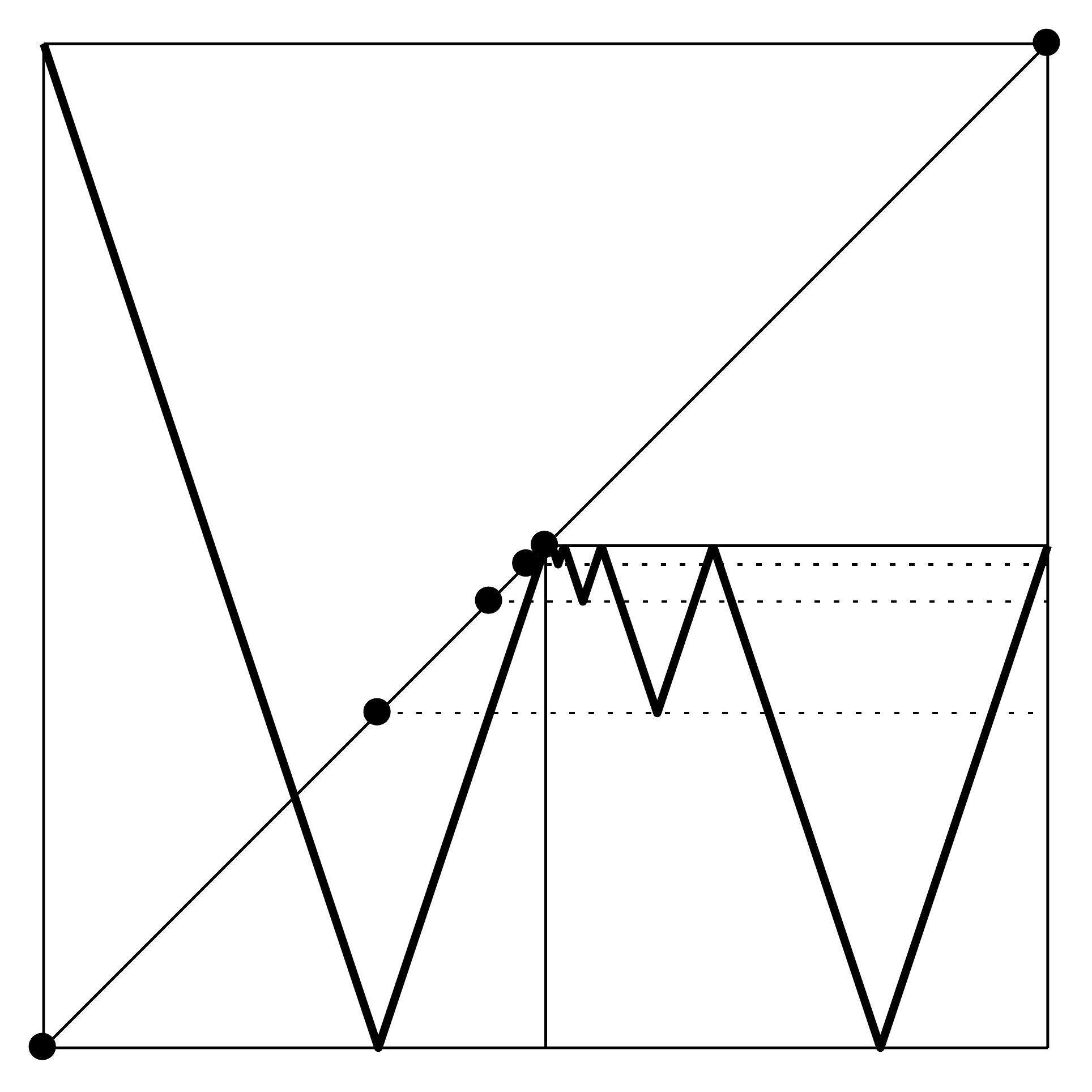}
\caption{The production of a finitely generated map}\label{fig:fgexample}
\end{figure}

We still need to verify that $f$ and $g$ are topologically mixing. We can do it using Theorem~\ref{th:mixing} below. For $f$ we use the taut Markov partition $\Crit(f)$, and for $g$ we use the taut Markov partition $P\cup\Crit(g)$. It is a simple matter to draw the transition graph $\Gamma(g,P\cup\Crit g)$ and verify that the vertex $(0,\frac13)$ is visited by every infinite path; this is our finite Rome.

What can we say about the dynamics of $g$? Observe that $g$ has constant slope $3$. Thus $g$ is its own constant slope model, and by Theorem~\ref{th:main} has no other constant slope models. Moreover, Theorem~\ref{th:main} shows that $g$ is Vere-Jones recurrent and has topological entropy $h(g)=\log 3$. Observe also that the fixed point $\tfrac12$, although strongly repelling for $f$, is only weakly repelling for $g$. This illustrates the sharpness of Theorem~\ref{th:rome}.

\subsection{Examples and nonexamples from the literature.}\strut\\

Table~\ref{tab:examples} compares and contrasts several countably monotone maps from the literature. Example~(a) is not locally eventually onto. Examples~(b) and~(c) both have fixed points which are not weakly repelling. It follows by Corollaries~\ref{cor:fgleo} and~\ref{cor:fgrep} that these maps are not finitely generated. The nonexistence / nonuniqueness of constant slope models for these maps gives a nice contrast to our theory.

Examples~(d) and~(e) are both generated by a global window perturbation of the tent map. By manipulating the number of laps in each window, one can control whether the resulting map is transient or recurrent \cite[\S7.2.1]{BB}.  The presence or absence of a constant slope model for these maps illustrates Theorem~\ref{th:main}.

\begin{remark}
The techniques used to construct examples~(d) and~(e) can be extended to show that both transient and recurrent maps are dense in the space of finitely generated maps with the $C^0$ topology.
\end{remark}

\begin{table}[htb!!]
\setlength{\tabcolsep}{0pt}
\renewcommand{\arraystretch}{1.1}
\caption{Examples and nonexamples of finitely generated maps}\label{tab:examples}
\vspace{-0.8em}
\begin{tabular}{Y{0.12} Y{0.14} Y{0.14} Y{0.14} Y{0.14} Y{0.14} Z}
\toprule
& (a) & (b) & (c) & (d) & (e) & \tabularnewline
Graph &
\includegraphics[width=0.12\textwidth]{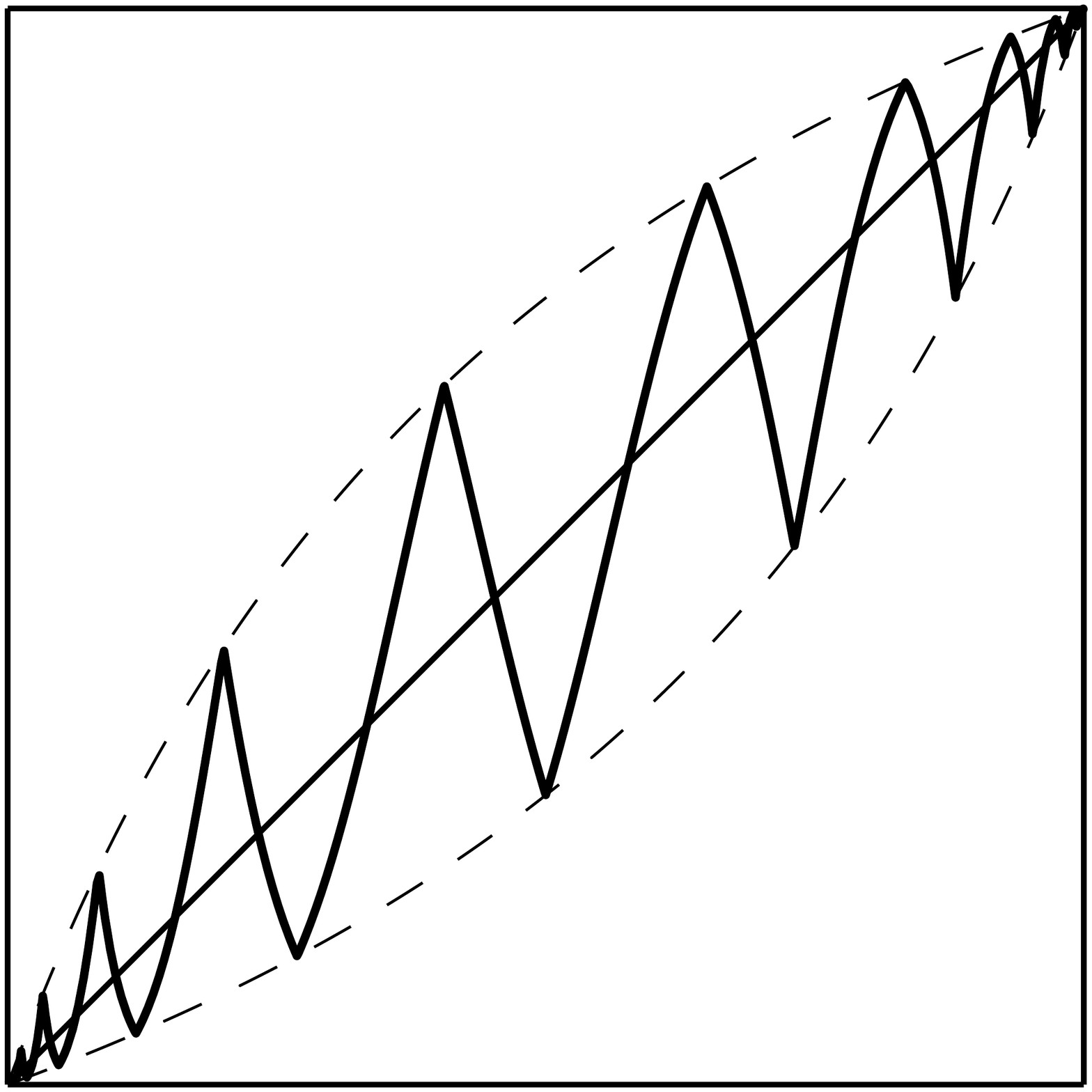} &
\includegraphics[width=0.12\textwidth]{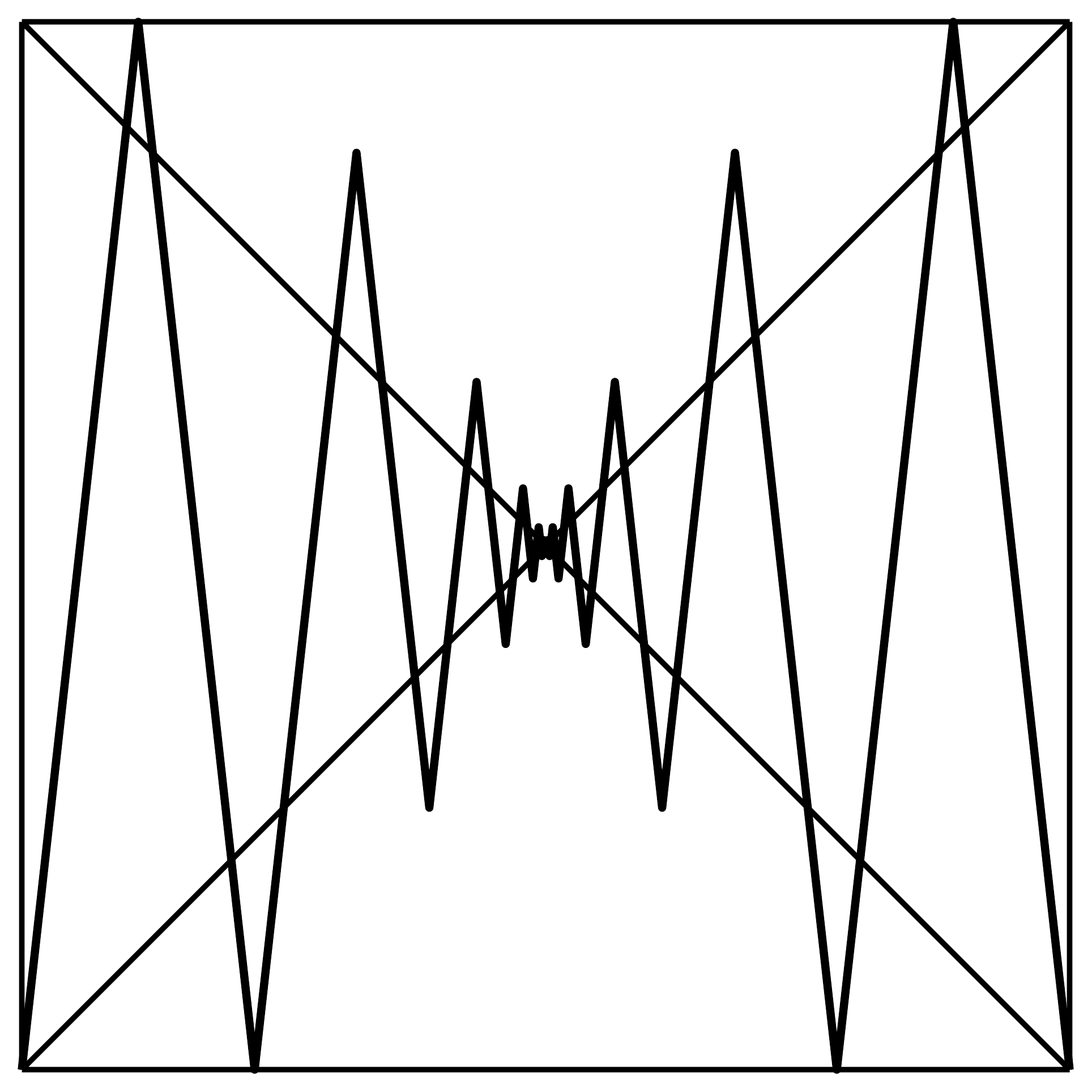} &
\includegraphics[width=0.12\textwidth]{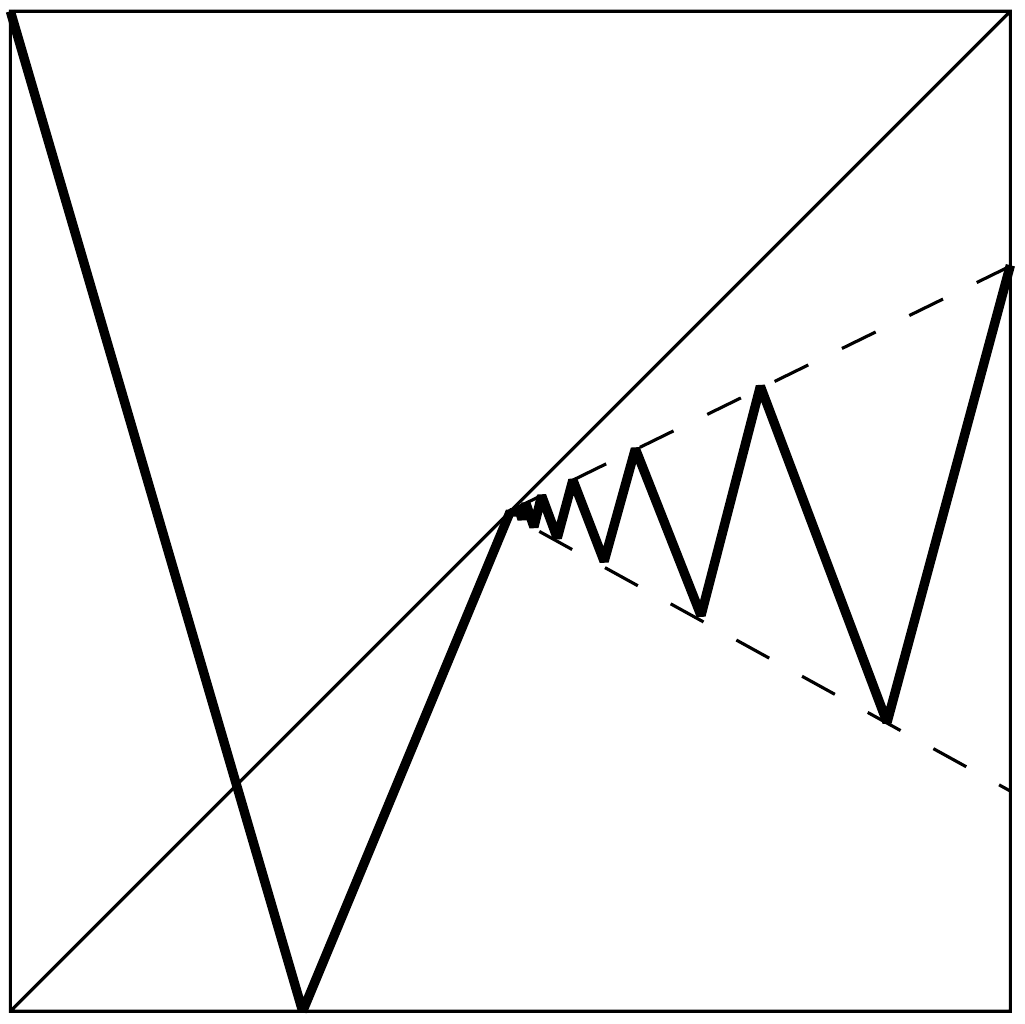} &
\includegraphics[width=0.12\textwidth]{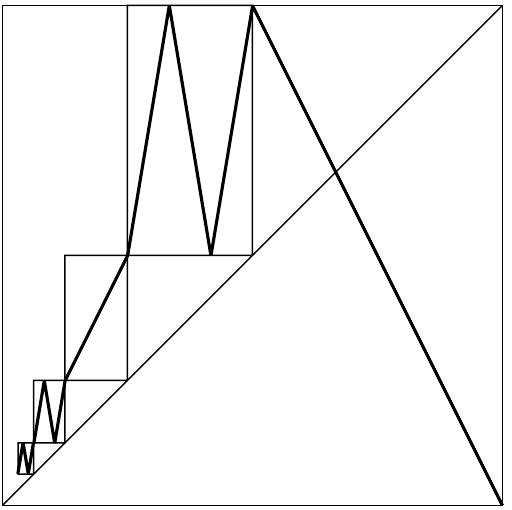} &
\includegraphics[width=0.12\textwidth]{fgtent.pdf} &
 \tabularnewline
Source & \cite[\S 7.1]{BB} & \cite[\S 7.2.4]{BB} & \cite[\S 8]{MR} & \cite[\S 7.2.1]{BB} & \cite[\S 7.2.1]{BB} & \tabularnewline
Vere-Jones classification & Recurrent & Transient & Recurrent & Recurrent & Transient & \tabularnewline
Finitely generated & No & No & No & Yes & Yes & \tabularnewline
Constant slope models & None & For all $\lambda\geq\exp h(f)$ & For all $\lambda\geq\exp h(f)$ & Unique, $\lambda=\exp h(f)$ & None &  \tabularnewline
\bottomrule
\end{tabular}
\end{table}

%In \cite[Section 7.2.1]{BB} the reader can find a family of finitely generated maps produced by global window perturbation of the full tent map. The family contains both recurrent and transient examples. Our Theorem~\ref{th:main} justifies and generalizes observations from that paper \cite[Remark 14]{BB}.

%Figure~\ref{fig:notfg} shows two maps from the same paper \cite[from Section 7.1 and Section 7.2.4]{BB} which are not finitely generated. The first map is not locally eventually onto, while the second has a fixed point which is not weakly repelling. The first map has (null) recurrent symbolic dynamics, but nevertheless has no constant slope model, because none of the eigenvectors of its transition matrix are summable. The second map, despite having transient symbolic dynamics, has constant slope models of slope $\lambda$ for every $\lambda$ greater than or equal to the exponential of the topological entropy (see \cite{BB} for details).

\section{How to Construct Finitely Generated Maps.}\label{sec:exa}

The hardest part of constructing finitely generated maps is meeting the topological mixing requirement. To that end we offer Theorem~\ref{th:mixing} and Corollary~\ref{cor:constructfg}.

\begin{definition}
Let $P$ be a taut Markov partition for $f$. We say that $f$ is \emph{$P$-affine} if $f|_I$ is affine for each $I\in\P$.
\end{definition}

\begin{theorem}\label{th:mixing}
Suppose $f\in C([0,1])$ is $P$-affine for some taut Markov partition $P$. If $A(f,P)$ is irreducible and aperiodic and if $\Gamma(f,P)$ contains a finite Rome $\R$, then $f$ is topologically mixing.
\end{theorem}

\begin{lemma}\label{lem:romevisits}
If a countable directed graph $\Gamma$ contains a finite Rome $\R$, then for every path $I_0\to I_1\to \cdots$ in $\Gamma$ there is a vertex $R\in\R$ visited by this path infinitely often.
\end{lemma}
\begin{proof}
If $n$ is a time when this path visits Rome $I_n\in\R$, then $I_{n+1}\to I_{n+2}\to \cdots$ is another path in $\Gamma$ which must also visit Rome.  Therefore our path visits Rome infinitely often, and since Rome is finite, our path must pass through one vertex $R\in\R$ infinitely often.
\end{proof}

\begin{proof}[Proof of Theorem~\ref{th:mixing}]
The essence of the proof is to show that $f$ has no homtervals, i.e. no non-degenerate interval $I$ such that $f^n(I)\cap\Crit f=\emptyset$ for all $n\in\mathbb{N}$.  Assuming no homtervals, it follows that any open set $U$ contains an open interval $U_0$ with some image $f^{n_0}(U_0)$ containing two critical points of $f$, and hence an entire interval $I\in\P$.  Given two more open sets $V,W$ we can find two more partition intervals which meet them $J\cap V\neq\emptyset$, $K\cap W\neq\emptyset$.  By irreducibility and aperiodicity we can find a natural number $n$ such that $A(f,P)^n_{IJ}>0$ and $A(f,P)^n_{IK}>0$ simultaneously.  Then $f^{n_0+n}(U)$ contains both $J$ and $K$ and hence meets both $V$ and $W$.  This proves that $f$ is mixing.

Now let us show that $f$ has no homtervals.  Assume to the contrary that $U$ is a homterval, and for each $n\in\mathbb{N}$ let $I_n\in\P$ be the partition interval containing $f^n(U)$. Then $I_0\to I_1\to \cdots$ is a path in $\Gamma$.  By Lemma~\ref{lem:romevisits}, it visits some vertex $R\in\R$ infinitely often.  Let $k<l$ be two arbitrary times when our path reaches this vertex, $I_k=I_l=R$.  Consider the corresponding sets 
\begin{align*}
U_k&:=I_k\cap f^{-1}(I_{k+1}) \cap f^{-2}(I_{k+2}) \cap \cdots\\
U_l&:=I_l\cap f^{-1}(I_{l+1}) \cap f^{-2}(I_{l+2}) \cap \cdots.
\end{align*}
These sets are homtervals contained in $R$; they are nondegenerate because $U_k$ contains $f^k(U)$ and $U_l$ contains $f^l(U)$. Also, $U_k$ and $U_l$ are either equal or disjoint, because if they are not equal, then taking $n$ minimal such that $I_{k+n}\neq I_{l+n}$ we have $f^n(U_k)\subseteq I_{k+n}$ and $f^n(U_l)\subseteq I_{l+n}$ disjoint.  The piecewise linearity of $f$ allows us to compare the lengths of $U_k$ and $U_l$,
\begin{equation}\label{lengths}
\begin{aligned}
|U_k| &= |I_k| \times \frac{|I_{k+1}|}{|f(I_k)|} \times \cdots \times \frac{|I_{l-1}|}{|f(I_{l-2})|} \times \frac{|I_l|}{|f(I_{l-1})|} \times \cdots = \\
&= \frac{|I_k|}{|f(I_{l-1})|} \times \frac{|I_{k+1}|}{|f(I_k)|} \times \cdots \times \frac{|I_{l-1}|}{|f(I_{l-2})|} \times |U_l|.
\end{aligned}
\end{equation}
Each of these fractions is less than or equal to $1$, since the interval in the denominator contains the interval in the numerator.  Therefore $|U_l|\geq|U_k|$.

Now we use the fact that $U_l$ and $U_k$ are either equal or disjoint.  If $U_l=U_k$, then each of the fractions in~\eqref{lengths} equals $1$, so the intervals in the numerators and denominators are equal.  Thus we have a loop $I_k \to I_{k+1} \to \cdots \to I_{k+l-1} \to I_l=I_k$ in $\Gamma(f,P)$ with no arrows pointing out of this loop, contradicting the irreducibility and aperiodicity of $A(f,P)$.  We are forced to conclude that $U_k$ and $U_l$ are disjoint homtervals with $U_l$ bigger than $U_k$.

Our construction of $U_k$ and $U_l$ used the choice of two numbers $k<l$ with $I_k=I_l=R$.  But our path reaches $R$ infinitely many times.  Thus we find infinitely many indices $n$ with $I_n=R$, resulting in infinitely many homtervals $U_n$ which are pairwise disjoint and have lengths bounded below by the length $|U_{n_0}|$, where $n_0$ is the minimal index with $I_{n_0}=R$.  But $[0,1]$ cannot contain an infinite collection of pairwise disjoint intervals with lengths bounded away from zero.  We are forced to conclude that there is no homterval $U$ for $f$.
\end{proof}

In order to easily apply Theorem~\ref{th:mixing} in the production of finitely generated maps, we offer one more observation.

\begin{lemma}\label{lem:slacktotaut}
Consider a map $g$ with a slack Markov partition $P$ and finitely many critical points in each partition interval $I\in\P$. Let $Q=P\cup\Crit(g)$.
\begin{itemize}
\item If $A(g,P)$ is irreducible and aperiodic, then so is $A(g,Q)$.
\item If $\Gamma(g,P)$ contains a finite Rome, then so does $\Gamma(g,Q)$.
\end{itemize}
\end{lemma}
\begin{proof}
Throughout the proof we use the map $\widehat{\cdot}:\Q\to\P$ which assigns to a partition interval $I\in\Q$ the unique partition interval $\widehat{I}\in\P$ which contains it.

Suppose $A(g,P)$ is irreducible and aperiodic. Fix two intervals $I,K\in\Q$. Find some $J\in\Q$ with $g(I)\supset J$. Notice that each endpoint of $I$ is either an endpoint of $\widehat{I}$ or a critical point of $g$, and therefore both of these endpoints have images in $P$. It follows that $g(I)\supset \widehat{J}$. By the irreducibility and aperiodicity of $A(g,P)$ there exists $N$ such that for all $n\geq N$, $g^n(\widehat{J})\supset\widehat{K}$. But then for all $n\geq N+1$, $g^n(I)\supset K$. This shows that $g(f,Q)$ is irreducible and aperiodic.

Now suppose $\R_P$ is a finite Rome in $\Gamma(g,P)$, and define $\R_Q$ to be the collection of all $I\in\Q$ such that $\widehat{I}\in\R_\P$. The hypotheses of the lemma give finiteness of $\R_Q$. If there is an arrow $I\to J$ in $\Gamma(g,Q)$, then $g(I)\supset J$, so $g(\widehat{I})\cap\widehat{J}\neq\emptyset$, which by the Markov property of $P$ implies that $g(\widehat{I})\supset\widehat{J}$ and there is at least one arrow $\widehat{I}\to\widehat{J}$ in $\Gamma(g,P)$. Thus, every path $I_0\to I_1\to \cdots$ in $\Gamma(g,Q)$ corresponds to at least one path $\widehat{I}_0\to \widehat{I}_1\to \cdots$ in $\Gamma(g,P)$. From this latter path find $n$ such that $\widehat{I}_n\in\R_P$. Then $I_n\in\R_Q$. Thus $\R_Q$ is a finite Rome.
\end{proof}

\begin{corollary}\label{cor:constructfg}
Let $g$ be a global window perturbation of a pair $(f,P)$ with $\Crit(f)$ finite, $\Acc P$ finite, and $f$ topologically mixing. Let $Q=P\cup\Crit(g)$. If $\Crit(g)\cap I$ is finite for each $I\in\P$ and if $g$ is $Q$-affine, then $g$ is finitely generated.
\end{corollary}
\begin{proof}
Since $f$ is topologically mixing, $A(f,P)$ is irreducible and aperiodic. By Lemma~\ref{lem:perrepel} and Theorem~\ref{th:rome}, $\Gamma(f,P)$ contains a finite Rome. By Lemma~\ref{lem:same}, $A(g,P)$ is irreducible and aperiodic and $\Gamma(g,P)$ contains a finite Rome. By Lemma~\ref{lem:slacktotaut}, $A(g,Q)$ is irreducible and aperiodic and $\Gamma(g,Q)$ contains a finite Rome. Now apply Theorem~\ref{th:mixing} to the map $g$ and the partition $Q$.
\end{proof}


\begin{thebibliography}{99}

\bibitem{AM}
Ll.~Alsed{\`a} and M.~Misiurewicz,
Semiconjugacy to a map of a constant slope.
\emph{Discrete Contin. Dyn. Syst. Ser. B} \textbf{20}~(2015), no. 10, 3403--3413.

\bibitem{B}
J.~Bobok,
Semiconjugacy to a map of a constant slope.
\emph{Studia Math.} \textbf{208}~(2012), no. 3, 213--228.

\bibitem{BB}
J.~Bobok and H.~Bruin,
Constant slope maps and the Vere-Jones classification.
\emph{Entropy} \textbf{18}~(2016), no. 6, Paper No. 234, 27 pp.

\bibitem{BS}
J.~Bobok and M.~Soukenka,
On piecewise affine interval maps with countably many laps.
\emph{Discrete Cont. Dyn. Syst.} \textbf{31}~(2011), no. 3, 753--762.

\bibitem{F}
W.~Feller,
\emph{An introduction to probability theory and its applications.}
3rd Ed. Vol. 1.
John Wiley \& Sons, Inc., New York, 1950.

\bibitem{H}
T.E.~Harris,
Transient Markov chains with stationary measures.
\emph{Proc. Amer. Math. Soc.} \textbf{8}~(1957), 937--942.

\bibitem{KH}
A.~Katok and B.~Hasselblatt.
\emph{Introduction to the modern theory of dynamical systems}.
Cambridge University Press, Cambridge, UK, 1995.

\bibitem{K}
B.~Kitchens,
\emph{Symbolic dynamics. One-sided, two-sided, and countable state Markov shifts.}
Universitext. Springer-Verlag, Berlin, 1988.

\bibitem{MR}
M.~Misiurewicz and S.~Roth,
No semiconjugacy to a map of constant slope.
\emph{Ergodic Theory Dynam. Systems} \textbf{36}~(2016), no. 3, 875--889.

\bibitem{P}
W.~Parry,
Symbolic dynamics and transformations of the unit interval.
\emph{Trans. Amer. Math. Soc.} \textbf{122}~(1966), 368--378.

\bibitem{Ru}
S.~Ruette,
\emph{Chaos on the interval}.
University Lecture Series, 67.
American Mathematical Society, Providence, RI, 2017.

\bibitem{VJ}
D.~Vere-Jones
Ergodic properties of nonnegative matrices--I.
\emph{Pacific J. Math.} \textbf{22}~(1967), no. 2, 361--386.


\end{thebibliography}
\end{document}